\documentclass[12pt]{amsart}
\newtheorem{theorem}{Theorem}[section]
\newtheorem{lemma}[theorem]{Lemma}
\newtheorem{proposition}[theorem]{Proposition}

\newtheorem{definition}[theorem]{Definition}
\newtheorem{remark}[theorem]{Remark}
\newtheorem{question}[theorem]{Question}

\numberwithin{equation}{section}

\begin{document}

\baselineskip=15pt

\title[Brauer group and rationality of moduli spaces]{Brauer group and
birational type of moduli spaces of torsionfree sheaves on a nodal curve}

\author[U. N. Bhosle]{Usha N. Bhosle}

\address{School of Mathematics, Tata Institute of Fundamental
Research, Homi Bhabha Road, Mumbai 400005, India}

\email{usha@math.tifr.res.in}

\author[I. Biswas]{Indranil Biswas}

\address{School of Mathematics, Tata Institute of Fundamental
Research, Homi Bhabha Road, Mumbai 400005, India}

\email{indranil@math.tifr.res.in}



\date{}

\begin{abstract}
Let $U^{'s}_{\mathbb L}(n,d)$ be the moduli space of stable
vector bundles of rank $n$ and fixed determinant $\mathbb L$ of
degree $d$ on a nodal curve $Y$. The moduli space of semistable
vector bundles of rank $n$ and degree $d$ will be denoted by
$U'_Y(n,d)$. We calculate the Brauer groups of $U^{'s}_{\mathbb
L}(n,d)$. We study the question of rationality of
$U^{'s}_{\mathbb L}(n,d)$ and $U'_Y(n,d)$.
\end{abstract}

\maketitle

\section{Introduction}

The rationality and Brauer groups of moduli spaces of 
stable vector bundles with a fixed rank and determinant 
on a smooth curve have been investigated over the years
(\cite{N}, \cite{N2}, \cite{KS}, \cite{BBGN} to name a few). 
Here we consider these questions for moduli spaces
associated to irreducible nodal curves.

Let $k$ be an algebraically closed field of characteristic zero.
Let $Y$ be an irreducible reduced curve, defined over $k$, with
at most ordinary nodes as singularities. Let $\{y_j\}_{j=1}^\ell$
be the
nodes of $Y$. The arithmetic genus of $Y$ will be denoted by
$g_Y$. Let $X\, \longrightarrow\, Y$ be the normalization of $Y$.
The genus of $X$, which is called the
geometric genus of $Y$, will be denoted by $g_X$.

For integers $n\, \geq\,1$ and $d$, let $U_Y(n,d)$ denote the moduli
space of semistable torsionfree sheaves on $Y$ of rank $n$ and degree
$d$.
Let $U'_Y(n, d)$ (respectively, $U^{'s}_Y(n, d)$) be the
moduli space of semistable (respectively, stable) vector bundles
over $Y$ of rank $n$ and degree $d$. So
$U'_Y(n, d)$ and $U^{'s}_Y(n, d)$ are Zariski open subsets of $U_Y(n,d)$.
Fix a line bundle ${\mathbb L}\,\longrightarrow\, Y$ of
degree $d$. Let $U'_{\mathbb L}(n, d)$ (respectively, $U^{'s}_{\mathbb
L}(n, d)$) be the moduli space of semistable
(respectively, stable) vector bundles $E\, 
\longrightarrow\, Y$ of rank $n$ with $\bigwedge^n E\,=\,
{\mathbb L}$. We recall that
$U^{'s}_{\mathbb L}(n,d)$ is a nonsingular quasiprojective variety. 
Let
$$
{\mathbb P} \,\longrightarrow\, U^{'s}_{\mathbb L}(n,d)\times Y
$$
be the universal projective bundle. It may be clarified that although a
universal vector bundle rarely exists, the universal projective bundle always
exists and it is unique up to a unique
isomorphism. For a point $x\,\in\, Y$, let
$$
{\mathbb P}_x\,\longrightarrow\, U^{'s}_{\mathbb L}(n,d)
$$
be the projective bundle obtained by restricting ${\mathbb P}$ 
to $U^{'s}_{\mathbb L}(n,d)\times\{x\}$.
 
Let $h$ be the greatest common divisor of $n$ and $d$.

We first compute the Brauer group of $U^{'s}_{\mathbb 
L}(n,d)$.

\begin{theorem}\label{brgp}
Assume that $g_X\,\ge\, 2$; also, if $g_X=2\,=\,n$, then
assume that $d$ is odd. Then the Brauer group 
$$Br(U^{'s}_{\mathbb L}(n,d)) \,\cong\, 
{\mathbb Z}/h{\mathbb Z}\, ,$$
and it is generated by the Brauer class of ${\mathbb P}_x$.
\end{theorem}

In course of the proof of the above theorem, we show that
$U^{'s}_{\mathbb L}(n,d)$ is simply connected (see Lemma
\ref{lem-b1}).

We then investigate the birational class of the moduli spaces.

{}From the construction of $U_Y(n,d)$ it follows that there is a natural
element $\psi_{n,d}$ in the Brauer group of the function field of
$U_Y(n,d)$ (see Definition \ref{d14}).

Following \cite{KS}, a rational map of varieties $\phi\, :\, A\, 
\dasharrow\, B$
will be called \textit{birationally linear} if $\phi$ is a dominant 
rational map with generic fiber rational, i.e., the function field 
$k(A)$ is purely transcendental over $k(B)$.

\begin{theorem}\label{birational}
Assume that $g_Y\,\ge\, 2$.
\begin{enumerate}
\item There exists a birationally linear map 
$$
\mu\,:\, U_Y(n,d) \,\dasharrow\, U_Y(h,0)
$$
such that $\mu^*(\psi_{h,0})\, =\, \psi_{n,d}$.
Thus $U_Y(n,d)$ is birational to 
 $U_Y(h,0) \times {\mathbb A}^{(n^2-h^2)(g-1)}$. 
 \item If $h\,=\,1$, then $U'_{\mathbb L}(n,d)$ 
is rational for any line bundle ${\mathbb L}$.
 \item If $g_Y=2=n$, then $U'_{\mathbb L}(n,d)$ 
is rational.

\end{enumerate}
\end{theorem}

If the geometric genus $g_X$ is zero, then the following
two theorems hold:

\begin{theorem}\label{rational21}
Let $Y$ be a rational nodal curve.
\begin{enumerate}
	\item The compactified Jacobian ${\overline J}_Y$ 
of $Y$ is rational.
\item If $g_Y\,=\,1$, then $U_Y(n,d)$ is rational.
\item If $g_Y\,=\, 2$, and
$h\,\le\, 4$, then $U_Y(n,d)$ is rational.
\item Assume that $g_Y\,\ge\, 2$, $n\, \geq\, 2$, and
$n$ is coprime to $d$. Then $U_Y(n,d)$ is rational.
\end{enumerate} 
\end{theorem}

\begin{theorem}\label{stablyrational}
Let $Y$ be a rational nodal curve with $g_Y\,\ge\, 1$. 
\begin{enumerate}
	\item For $h$ dividing $420= 3.4.5.7$, 
the moduli space $U_Y(n,d)$ is stably rational.
  \item If $h$ a prime number, then $U^{'s}_Y(n,d)$ is retract rational. 
\end{enumerate}
\end{theorem}

\noindent
\textbf{Acknowledgements.} We thank N. Hoffman for pointing out the
reference \cite{BB}. We also thank the referee for helpful comments.

\section{The Brauer group}

We continue with the notation set up in the introduction.
In this section, we assume that $g_X\, \geq\, 2$. 
If $g_X\, =\, 2\, =\, n$, then we
assume that $d$ is odd. Our first aim is to prove
Theorem \ref{brgp}.

\subsection{Generator of the Brauer group}\label{cyclic}

We first recall some results from \cite{B3}. 
Without loss of generality, we may assume that 
$d>>0$ (by tensoring with a line bundle of sufficiently
large degree).
Then there is a projective space ${\mathbb P}^N$ with a family 
of vector bundles
\begin{equation}\label{eb2}
{\mathcal E}^0\, \longrightarrow\, {\mathbb P}^N\times Y\, ,
\end{equation}
of rank $n$ and with a fixed determinant $\mathbb{L}$ of degree $d$ 
such that for all $t\in {\mathbb P}^N$, the vector bundle
$\mathcal{E}^0_t\, :=\, {\mathcal E}^0\vert_{\{t\}\times Y}$
contains the trivial bundle of rank $n-1$ as subbundle; also,
all stable vector bundles of rank $n$ and determinant
$\mathbb L$ occur in this family
(see the proof of \cite[Proposition 2.3]{B3} for details). 
Let 
\begin{equation}\label{eb1}
{\mathcal U}_P\, \subset\, {\mathbb P}^N
\end{equation}
be the nonempty Zariski open subset corresponding 
to stable vector bundles in the family 
${\mathcal E}^0$. There is a natural surjective morphism
\begin{equation}\label{qu}
{\mathcal U}_P\, \longrightarrow\, {U'}^s_{\mathbb L}(n,d)
\end{equation}
(see \cite[p. 256]{B3}); the homomorphism
$\text{Pic}({U'}^s_{\mathbb L}(n,d))\, \longrightarrow\, 
\text{Pic}({\mathcal U}_P)$ induced by the morphism in \eqref{qu}
is in fact an injection, and under the above assumptions on the 
geometric genus, we have
$$
\text{Pic}({U'}^s_{\mathbb L}(n,d))\,=\,\text{Pic}
({\mathcal U}_P)\,=\, {\mathbb Z}
$$
\cite[Theorem I(1)]{B3}. Consequently, the codimension of the
complement ${\mathbb P}^N\setminus {\mathcal U}_P$ 
is strictly greater than one.

Fix a smooth point $y\, \in\, Y$. Let
$$
{\mathbb P}_{\rm univ}\, \longrightarrow\, {U'}^s_{\mathbb 
L}(n,d)\times Y
$$
be the universal projective bundle. Let
\begin{equation}\label{e0}
f\, :\, {\mathbb P}_y\, :=\, {\mathbb P}_{\rm univ}\vert_{
{U'}^s_{\mathbb L}(n,d)\times\{y\}}\,\longrightarrow\, {U'}^s_{\mathbb L}(n,d)
\end{equation}
be the projective bundle obtained by restricting ${\mathbb P}_{\rm 
univ}$.

There exists a natural universal vector bundle
\begin{equation}\label{eb3}
{\mathcal E}\, \longrightarrow\, {\mathbb P}_y\times Y\, .
\end{equation}
The following argument showing that $\mathcal E$ exists is due to
the referee.

Realize ${\mathbb P}_y$ as parabolic moduli space with parabolic structure
at $y$; there is exactly one nonzero parabolic weight, and its multiplicity 
is one. Therefore, there is a universal parabolic vector bundle over
${\mathbb P}_y\times Y$ \cite[p. 465, Proposition 3.2]{BY}. The vector bundle
$\mathcal E$ in \eqref{eb3} can be taken to be the vector
bundle underlying a universal parabolic bundle on ${\mathbb P}_y\times Y$.

Let
\begin{equation}\label{ph2}
\phi\, :\, {\mathbb P}_y\times Y\,\longrightarrow\,
{\mathbb P}_y
\end{equation}
be the natural projection.
For $u\,\in\, {\mathbb P}_y$, the vector bundle
${\mathcal E}\vert_{u\times Y}$ is represented by
$f(u) \,\in\, U_L^s(n,d)$. As $d>>0$ (this assumption was
imposed), the direct image 
$\phi_*{\mathcal E}\, \longrightarrow\,
{\mathbb P}_y$ is locally free, hence
$\phi_*{\mathcal E}$
is trivial on a nonempty Zariski open subset
$$
\widetilde{U}\, \subset\, {\mathbb P}_y\, .
$$
Using this trivialization, we have
${\mathcal O}_U^{n-1} \subset {\mathcal E}\vert_{U\times Y}$
on some suitable nonempty Zariski open subset $U\, \subset\,
\widetilde{U}$. Thus there is a morphism
$$
\psi\, :\, U\, \longrightarrow\, {\mathcal U}_P\, ,
$$
where ${\mathcal U}_P$ is the open subset in \eqref{eb1},
such that the family ${\mathcal E}\vert_{U\times Y}$ (see
\eqref{eb3}) is isomorphic to the pulled back family
$$
(\psi\times\text{Id}_Y)^*{\mathcal E}^0\, \longrightarrow\,
U\times Y
$$
(see \eqref{eb2} for ${\mathcal E}^0$).

We have a commutative diagram
\begin{equation}\label{eb4}
\begin{matrix}
U & \hookrightarrow & {\mathbb P}_y\\
~\Big\downarrow\psi && ~ \Big\downarrow f\\
{\mathcal U}_P& \longrightarrow & {U'}^s_{\mathbb L}(n,d)\\
\end{matrix}
\end{equation}
($f$ is defined in \eqref{e0}). Let
\begin{equation}\label{fb0}
\begin{matrix}
Br({U'}^s_{\mathbb L}(n,d)) & \longrightarrow &
Br({\mathcal U}_P)\\
~ \Big\downarrow a && \Big\downarrow \\
Br({\mathbb P}_y) & \stackrel{b}{\longrightarrow} & Br(U)
\end{matrix}
\end{equation}
be the corresponding diagram of homomorphisms of Brauer groups.
Since ${\mathcal U}_P$ is a Zariski open subset of 
${\mathbb P}^N$ such that the codimension of the
complement ${\mathbb P}^N\setminus {\mathcal U}_P$ is at least two,
we have $Br({\mathcal U}_P)\,=\, Br({\mathbb P}^N)\,=\,
0$. Consequently, from \eqref{fb0} we conclude
that $b\circ a\,=\, 0$.
The homomorphism $b$ is injective because
$U$ is a Zariski open subset of ${\mathbb P}_y$.
Therefore,
\begin{equation}\label{ka}
Br({U'}^s_{\mathbb L}(n,d))\, =\, \text{kernel}(a)\, .
\end{equation}

Let
\begin{equation}\label{eb6}
\beta\, \in\, Br({U'}^s_{\mathbb L}(n,d))
\end{equation}
be the class of
the projective bundle ${\mathbb P}_y\,\longrightarrow\,
{U'}^s_{\mathbb L}(n,d)$. It is known that the kernel of the
homomorphism $a$ in \eqref{fb0} is generated by $\beta$
(see \cite[p. 193, Theorem 2]{Ga}).

\subsection{Period of the Brauer class}

We will prove that the order of the class $\beta$ in \eqref{eb6}
is the greatest common divisor of $n$ and $d$.

\begin{lemma}\label{lem-b2}
The order of $\beta$ (see \eqref{eb6}) divides
$h\,:=\,{\rm g.c.d.}(n,d)$.
\end{lemma}

\begin{proof}
Let $\mu_n$ be the cyclic group
consisting of the $n$--th roots of unity.
Since $\beta$ is the class of a projective bundle of relative
dimension $n-1$, it lies in the image of the homomorphism
$$
H^2({U'}^s_{\mathbb L}(n,d),\, \mu_n)\,\longrightarrow\,
H^2({U'}^s_{\mathbb L}(n,d),\, {\mathbb G}_m)\, .
$$
Hence the order of $\beta$ divides $n$.

Let ${{\mathcal U}'}^s_{\mathbb L}(n,d)$ be the moduli stack of 
stable vector bundles $E\, \longrightarrow \,Y$ of rank $n$
with $\bigwedge^n E\,=\, {\mathbb L}$.
Let
\begin{equation}\label{eb7}
{{\mathcal U}'}^s_{\mathbb L}(n,d)\, \longrightarrow\,
{U'}^s_{\mathbb L}(n,d)
\end{equation}
be the coarse moduli space. We note that
${{\mathcal U}'}^s_{\mathbb L}(n,d)$ is a gerbe over
${U'}^s_{\mathbb L}(n,d)$ banded by the group $\mu_n$.
We recall that a gerbe on ${U'}^s_{\mathbb L}(n,d)$
defines an element of $Br({U'}^s_{\mathbb L}(n,d))$
which can be constructed as follows. For any vector bundle
$W$ on the gerbe, the endomorphism bundle $End(W)$ descends
to ${U'}^s_{\mathbb L}(n,d)$ as an Azumaya algebra bundle.
The Brauer class of this Azumaya algebra bundle is independent
of the choice of $W$. (A vector bundle on a $\mu_n$--gerbe
over ${U'}^s_{\mathbb L}(n,d)$ is also called a \textit{twisted
vector bundle} on ${U'}^s_{\mathbb L}(n,d)$.)

The restriction to ${{\mathcal
U}'}^s_{\mathbb L}(n,d)\times\{y\}$ of the
universal vector bundle
$$
{\mathcal E}'\, \longrightarrow\, {{\mathcal
U}'}^s_{\mathbb L}(n,d)\times Y\, .
$$
is a vector bundle on the gerbe ${{\mathcal U}'}^s_{\mathbb L}(n,d)$.
So, the class of the gerbe in \eqref{eb7} is $\beta$.

Let $q_i$, $i\,=\,1\, ,2$, be the natural projection of
${{\mathcal U}'}^s_{\mathbb L}(n,d)\times Y$ to the $i$--th
factor. Since $d>>0$, we know that $H^1(Y,E)\,=\, 0$ for all $E\,\in\,
{U'}^s_{\mathbb L}(n,d)$. Therefore, the direct image
bundle
\begin{equation}\label{cV}
{\mathcal V}\, :=\, q_{1*}({\mathcal E}')
\, \longrightarrow\,{{\mathcal U}'}^s_{\mathbb L}(n,d)
\end{equation}
is a vector bundle on the gerbe. Hence, the Azumaya algebra bundle
$End({\mathcal V})$ on ${U'}^s_{\mathbb L}(n,d)$ is represented
by $\beta$. Consequently, the order of $\beta$ divides the rank of 
${\mathcal V}$. From Riemann--Roch, the rank of
${\mathcal V}$ is $d+n(1-g_Y)$.
Hence the order of $\beta$ divides $d$ (recall that it divides $n$).
Therefore, the order of $\beta$ divides $h$.
\end{proof}

\subsection{Proof of Theorem \ref{brgp}}
Express ${U'}^s_{\mathbb L}(n,d)$ as a GIT quotient
$R^{s}/\!/\text{GL}_q(k)$ in the standard way.
It is known that there is a $\text{GL}_q(k)$--linearized
line bundle over $R^{s}$ with the center ${\mathbb G}_m\,\subset\,
\text{GL}_q(k)$ acting as multiplication by $t^j$,
$t\, \in\, {\mathbb G}_m$, if and only if $j$ is a multiple of
$h$ (the proof given in \cite{DN} for smooth curves goes through
without any change). Therefore for $h>1$, if $c\, \in\, [1\, , h-1]$, 
then $c\beta\, \not=\, 0$. This implies that
the order of $\beta$ is not smaller than $h$. 
Now from Lemma \ref{lem-b2} we conclude that the
order of $\beta$ is $h$. This, together with the earlier
established fact that $\beta$ is a generator 
of $Br({U'}^s_{\mathbb L}(n,d))$,
completes the proof of Theorem \ref{brgp}.

\subsection{Simply connectedness}

\begin{lemma}\label{lem-b1}
The variety ${U'}^s_{\mathbb L}(n,d)$ is simply connected.
\end{lemma}

\begin{proof}
Consider the diagram of homomorphisms of fundamental groups
\begin{equation}\label{fb1}
\begin{matrix}
\pi_1(U) & \stackrel{p}{\longrightarrow} & \pi_1({\mathbb P}_y)\\
~\Big\downarrow q && ~\Big\downarrow r\\
\pi_1({\mathcal U}_P) & \longrightarrow &
\pi_1({U'}^s_{\mathbb L}(n,d))\\
\end{matrix}
\end{equation}
given by \eqref{eb4}. First note that
$\pi_1({\mathcal U}_P)\,=\,\pi_1({\mathbb P}^N)
\,=\, e$ because the codimension of ${\mathbb P}^N\setminus
{\mathcal U}_P$ is at least two. Hence
$$
r\circ p\,=\, 0\, .
$$
The homomorphism $p$ is surjective because $U\, \subset \,
{\mathbb P}_y$ is Zariski open. Since the fibers of the projection
$f$ in \eqref{e0} are $1$--connected, from the long exact sequence
of fundamental groups for the fibration $f$ we conclude that the
homomorphism $r$ in \eqref{fb1} is an isomorphism. Hence we
conclude that ${U'}^s_{\mathbb L}(n,d)$ is simply connected.
\end{proof}

\section{Birational type of moduli spaces}

The explicit descriptions of the moduli spaces for 
low ranks and genera enable us to determine their birational 
classification. 

\subsection{Cases of low ranks and genera}

\begin{lemma}\label{l-3.1}
The moduli space $U_Y(n,d)$ of semistable torsionfree sheaves 
of rank $n$ and degree $d$ on a rational nodal 
curve $Y$ of arithmetic genus $1$ is rational.
\end{lemma}

\begin{proof}
Let $S^h(Y)$ be the $h$--th symmetric product of $Y$, in other
words, 
$$
S^h(Y)\,= \,\prod_{i=1}^h Y_i/\Sigma_h\, ,
$$
where $Y_i$ is a copy of $Y$, and $\Sigma_h$ is the 
group of permutations of $\{1,\cdots ,h\}$. 
The moduli space $U_Y(n,d)$ is isomorphic to $S^h(Y)$,
where $h\, =\, \text{g.c.d.}(n,d)$ \cite[Remark 1.27]{HLST}.
Hence it suffices to show that $S^h(Y)$ is rational.

Since $Y$ is birational to ${\mathbb P}^1_k$, the symmetric
product $S^h(Y)$ is birational to $S^h({\mathbb P}^1_k)$.
But $S^h({\mathbb P}^1_k)\, =\, |{\mathcal O}_{{\mathbb P}^1_k}(h)|
\,=\, {\mathbb P}^h_k$. Hence the lemma follows.
\end{proof}

\begin{lemma}
If $g_Y\,=\,2$ and $n\,=\,2$, then the moduli space
$U^{'s}_{\mathbb L}(n,d)$ is rational. 
\end{lemma}

\begin{proof} 
If $g_Y\,=\,2\,=\, n$, and $d$ is even, then $U^{'s}_{\mathbb L}(n,d)$ 
is an open subset of ${\mathbb P}^3$ \cite{B2}. Hence it is rational. 

If $g_Y\,=\,2\,=\, n$, and $d$ is odd, then
$U^{'s}_{\mathbb L}(n,d)$ is rational. Indeed, this is a 
special case of Theorem \ref{birational}(2) which will be
proved later.
\end{proof}

\subsection{The case of higher genera}

Assume that $g_Y\,\ge\, 2$ and $n\,\ge\, 2$.
If $Y$ is nonsingular,
Theorem \ref{birational}(1) is due to
King and Schofield \cite{KS}. Their proof essentially 
goes through in the nodal case.
To avoid extensive repetition, we mainly explain 
the modifications needed in the proof for the nodal case.

We recall the definitions and constructions used
in \cite{KS}. 
For a torsionfree sheaf $E$ of rank $r(E)$, degree $d(E)$, we call 
$(r(E),d(E))$ the type of $E$. Let $\mu(E)$ denote the slope of $E$. 

For torsionfree sheaves $E, F$ with
at least one of $E$ and $F$ locally free, we can define 
$$
\chi(F,E)\,=\, \sum_i (-1)^i \ \dim {\rm Ext}^i(F,E)\, .
$$
Note that the sum on the right hand side is infinite if both $E$ 
and $F$ are not locally free \cite{B5}. If 
at least one of $E$ and $F$ is locally free, then 
${\rm Ext}^i(F,E)= H^i(F^*\otimes E)$, and the terms in
the right side vanish except for $i\,=\, 0\, ,1$.
Then we have
\begin{equation}\label{be1}
\chi((r(F),d(F)), (r(E),d(E)))\, :=\,
\chi(F,E)= r(F)d(E)-r(E)d(F) -r(E)r(F)(g_Y-1)\, .
\end{equation}
Note that $\chi(F,E)$ is bilinear for exact sequences.

The following lemma is similar to Lemma 2.1 of \cite{KS}.

\begin{lemma}\label{l3.1}
Let $E$ and $F$ be general vector bundles of fixed types.
Suppose that
$${\rm Hom}\, (F,E) \,\neq\, 0\, ;$$
let 
$$\phi: F \longrightarrow E$$ be a general element. Then 
${\rm Ext}^1(F,E)=0$ and $\phi$ has maximal rank. 
If $r(E)\neq r(F)$, then ${\rm coker}(\phi)$ is torsionfree; 
in particular, $\phi$ is surjective (respectively, injective) 
if $r(E)\,<\,r(F)$ (respectively, $r(E)\,>\,r(F)$). 
\end{lemma}

\begin{proof}
A general bundle $E$ of type $(n,d)$ depends on 
$n^2(g-1)+1= 1-\chi(E,E)$ parameters, and similarly,
a general $F$ depends on $1-\chi(F,F)$ parameters. 
The homothety class $[\phi]\in {\mathbb P}( \ {\rm Hom} \ (F,E))$ 
of $\phi$ depends on $h^0(F^*\otimes E)-1$ parameters. 
Thus the triple $(F,E,[\phi])$ depends on 
\begin{equation}\label{ba2}
p_0\, :=\, 1- \chi(F,F)-\chi(E,E)+h^0(F^*\otimes E)
\end{equation}
parameters.

Define $K:= {\rm ker} \ \phi$, $I:= {\rm im} \ \phi$
and $Q:= {\rm coker} \ \phi$. Let 
$T:= T(Q)$ be the torsion subsheaf of $Q$. Define $Q':= Q/T(Q)$ 
and
$$I':= {\rm ker} (E \longrightarrow Q')\, .
$$
Since $E, F$ are general bundles, 
by \cite[Proposition 2.8(2)]{B5}, the torsionfree sheaves 
$K$, $I$, $I'$ and $Q'$ are all general vector bundles. 
The proof now proceeds exactly as for \cite[Lemma 2.1]{KS}.
\end{proof}

For a given type $(n,d)$, let $(s,e)$ be the unique type
satisfying the following two conditions:
\begin{itemize}
\item $\chi((s,e), (n,d)) = h:= \text{g.c.d.}(n,d)$, and 
\item if $h<r$, then $r/h <s<2r/h$, and if $h=r$, then $s=2$.
\end{itemize}

\begin{proposition}\label{p3.2}
There exists a vector bundle $F$ of type $(s,e)$ such that 
for a general vector bundle $E$ of type $(n,d)$ the following
hold:
\begin{enumerate}
\item $h^1(F^*\otimes E)=0$, $h^0(F^*\otimes E)=h$,
\item the natural map $\epsilon_F(E): \ {\rm Hom}(F,E)
\otimes_k F \longrightarrow E$ is surjective, and
\item $E_1\,=\, {\rm kernel}(\epsilon_F(E))$ is a general vector bundle 
with $h^1(E_1^*\otimes F)=0$.
\end{enumerate}
\end{proposition}

\begin{proof}
The proof of \cite[Proposition 2.3]{KS} goes through 
using Lemma \ref{l3.1}.
\end{proof}

All the three conditions in Proposition \ref{p3.2} are Zariski
open. Hence the general vector bundle $F$ satisfies the conditions
in Proposition \ref{p3.2}.

We now use Proposition \ref{p3.2} to define the Brauer class of 
(the function field of) the moduli space $U_Y(n,d)$. 
Note that $U_Y(n,d)$ is irreducible \cite{R}.

We recall the definition of Brauer class of a variety \cite[Definition 
3.1]{KS}.

\begin{definition}\label{bclass}
{\rm Let $Z$ be an affine algebraic variety on which the 
algebraic group $\text{PGL}_n(k)$
is acting freely. Over the quotient variety 
$Z/{\rm PGL}_n(k)$, there is a bundle of central simple algebras 
$M_n(k)\times^{\text{PGL}_n(k)} Z$ of dimension $n^2$. 
At the generic point of $Z/\text{PGL}_n(k)$, this is a central simple algebra 
over the function field $k(Z/\text{PGL}_n(k))$. This algebra defines a class 
in the Brauer group of $k(Z/\text{PGL}_n(k))$ which 
will be denoted by $\beta(Z/\text{PGL}_n(k))$.} 
\end{definition}
 
\subsection{The Brauer class of the moduli space $U_Y(n,d)$}\label{X(n,d)}

Fix a general vector bundle $F$ as in Proposition \ref{p3.2}. 
Let $X_{n,d}$ be the open subset of Quot $(F^h,n,d)$ 
parametrizing the quotients $q: F^h \longrightarrow E$ such that
\begin{itemize}
\item $E$ 
is a stable vector bundle of rank $n$ and degree $d$, and 

\item the induced map $k^{\oplus h}\, \longrightarrow\,
{\rm Hom} (F,E)$ is an isomorphism. 
\end{itemize}
The group $\text{PGL}_h(k)= P(\text{Aut}( F^{\oplus h}))$ acts on $X_{n,d}$ 
and the quotient $X_{n,d}/\!/\text{PGL}_h(k)$ is isomorphic to 
an open dense subset of $U_Y(n,d)$. Since $U_Y(n,d)$ is a projective
variety, taking the inverse image of an affine open dense subset of
$U_Y(n,d)$ contained in the image of $X_{n,d}$, we may replace $X_{n,d}$
by an affine open dense $\text{PGL}_h(k)$--equivariant subset of it.
Thus we can identify the function field of
$U_Y(n,d)$ with the function field $k(X_{n,d}/\!/\text{PGL}_h(k))$. 

\begin{definition}[\cite{KS}, Definition 3.3]\label{d14}

{\rm We define the} Brauer class 
$$\psi_{n,d}\, \in\, Br(k(U_Y(n,d))$$
{\rm as the Brauer class $\beta(X_{n,d}/\!/\text{PGL}_h(k))$ (see 
Definition \ref{bclass}).}
\end{definition}

\begin{theorem}\label{birational2}
Let $Y$ be an irreducible nodal curve with $g_Y\ge 2$.
\begin{enumerate}
	\item There exists a birationally linear map 
$$\mu: U_Y(n,d) \,\dasharrow\, U_Y(h,0) \ {\rm such \ that} \ 
\mu^*(\psi_{h,0})= \psi_{n,d}. $$
 Thus $U_Y(n,d)$ is birational to 
 $U_Y(h,0) \times {\mathbb A}^{(n^2-h^2)(g-1)}$. 
 \item If $n$ and $d$ are coprime, then $U^{'s}_{\mathbb L}(n,d)$ 
 is rational for any line bundle ${\mathbb L}$.
\end{enumerate}
\end{theorem}

\begin{proof}
In \cite[Theorem 5.1]{KS} this was proved for nonsingular $Y$.
The proof of \cite{KS} goes through for nodal $Y$ without
any change. In the nodal case, we consider only the parabolic 
moduli spaces with parabolic structure at smooth points (see
\cite[section 4]{KS}).
\end{proof}

This completes the proof of Theorem \ref{birational}.

\section{Moduli spaces of sheaves over rational nodal curves}

Throughout this section, $Y$ denotes a rational nodal curve 
(unless otherwise stated).
 
We first consider the case where $h=1$.

\begin{theorem}\label{rational1}
Let $Y$ be a rational nodal curve of genus $g_Y\ge 1$.
\begin{enumerate}
\item The compactified Jacobian ${\overline J}_Y$ 
of $Y$ is rational. 
\item If $n\ge 2$, and $n$ is coprime to $d$, then 
$U_Y(n,d)$ is rational.
\end{enumerate} 
\end{theorem}

\begin{proof}
A natural desingularization of the Jacobian of $Y$ is 
the $m$-fold product ($m$ being the number of nodes)
$$
{\widetilde J}_Y\,=\, {\mathbb P}^1 \times \cdots \times {\mathbb P}^1
$$
(see \cite{I}). Thus $ {\widetilde J}_Y$ and $\overline{J}_Y$ 
are rational.

If $n\geq 2$ and $d$ are coprime,
by Theorem \ref{birational2},
the moduli space $U_Y(n,d)$ is birational to 
 ${\overline J}_Y \times {\mathbb A}^{(n^2-1)(g-1)}$. 
The latter is rational by part (1). 
\end{proof}

Parts (1), (2) and (4) of Theorem \ref{rational21} now follow from Theorem 
\ref{rational1} and Lemma \ref{l-3.1}. Part (3) will be proved in the last section.

\noindent
\textbf{Assumptions:}
\begin{enumerate}
\item In view of Theorem \ref{rational1}, henceforth we assume that 
$h\ge 2$. 

\item In view of Lemma \ref{l-3.1}, we may, and we will,
assume that $g_Y\,\ge \, 2$.
\end{enumerate}

\begin{remark}\label{reduction}
{\rm By Theorem \ref{birational}, the moduli
space $U_Y(n,d)$ is birational to the Cartesian product
$U_Y(h,0) \times {\mathbb A}^{(n^2-h^2)(g_Y-1)}$. 
Hence to show that $U_Y(n,d)$ is rational (respectively 
stably rational), it is enough to prove the same for $U_Y(h,0)$.}
\end{remark}

We shall study rationality and stable rationality of $U_Y(h,0)$. 
Using generalized parabolic sheaves,
 we reduce the problems to problems on quotients 
of the products of $M_h$ (the ring of $h\times h$ matrices) by 
the conjugation action of $\text{PGL}_h$ (see Question \ref{c4.3}).
These problems are studied by several algebraists and algebraic 
geometers. Answers to them are known in several cases, 
most problems are still unsolved. 

\subsection{Generalized parabolic sheaves}\label{gps}

Let $X$ be a curve with at most ordinary nodes as singularities. 
Fix a positive integer $m$, and also fix divisors on $X$
$$D^j\,=\, x_j+z_j\, ,$$ $j\,\in\, [1\, , m]$, where 
$(x_j,z_j)$ are pairs of distinct smooth points 
of $X$. Define $D_m\,:=\, \{D^j\}_{j=1}^m$. Denote by 
$$M_{X,D_m}=M_{X,D_m}(n,d)$$ 
 the moduli space of stable generalized parabolic 
sheaves (``GPS'' for short) on $X$ with generalized 
parabolic structure over the divisors 
$D^j$, $j\,=\,1,\cdots ,m$ (see \cite{B1} for generalities on GPS). 
Let $X_m$ be the nodal curve obtained
by identifying the point $x_j$ with $z_j$ for each $j$.
For $j\,\in\, [1\, , m]$, let $y_j$ be the image of
$x_j$ (or $z_j$) in $X_m$, which is a singular point
of $X_m$. There is a birational surjective morphism 
\begin{equation}\label{Phi}
\Phi\,:\, M_{X,D_m}\, \longrightarrow\, U_{X_m}(n,d)
\end{equation}
\cite[Theorem 2]{B1}.

Let $X$ be a rational nodal curve and let 
$$p\,: \,{\mathbb P}^1 \,\longrightarrow\, X$$
be the normalization map. 

\begin{lemma}\label{l4.1}
Let $F$ be a stable vector bundle of rank $n$, degree $0$ 
over a rational curve $X$ of genus $g_X$ 
and define $E\,:= \,p^*F$. Then 
$E= \bigoplus _{i=1}^n {\mathcal O}_{{\mathbb P}^1}(a_i)$, 
where $-g_X+1\, \leq\, a_i\, \leq\, g_X-1$ with $\sum a_i\,=\,0$. 
\end{lemma}

\begin{proof}
A vector bundle $F$ of rank $n$, degree $0$ 
on $X$ gives the vector bundle $E= p^*F$ of rank $n$, 
degree $0$ on ${\mathbb P}^1$ together with isomorphisms 
$g_j: E_{x_j} \,\stackrel{\sim}{\longrightarrow}\, E_{z_j}$, $j=1,\cdots, g_X$
(recall that $p(x_j)\,=\, p(z_j)$).
Let $F_j(E)$ denote the graph of $g_j$. Then $(E, \{F_j(E)\})$ is 
the generalized parabolic sheaf 
(GPS) on ${\mathbb P}^1$ corresponding to $F$. 
For any subbundle $E_1 \subset E$, and $1\le j\le g_X$, 
define
$$F_j(E_1)\,:=\, F_j(E)\cap (E_{1,x_j}
\oplus E_{1,z_j}) \subset E_{x_j}\oplus E_{z_j}; 
\ f_j(E_1)= \ {\rm dim} \ F_j(E_1)\, .$$ 
By dimension count, one has 
\begin{equation}\label{(0)}
f_j(E_1) \,\ge\, r(E)+ 2r(E_1)-2r(E)= 2r(E_1)-r(E)\, .
\end{equation} 
Clearly, $f_j(E_1)\,\ge\, 0$. 
 
Since $F$ is a stable vector bundle, we know
that $(E, \{F_j(E)\})$ is a stable GPS. 
Therefore, 
for any subbundle $E_1 \subset E$, we have 
$$\mu(E_1) + (\sum_j f_j(E_1))/r(E_1) \,<\, g_X\, .$$

Let $E\,=\, \bigoplus _{i=1}^n {\mathcal O}_{{\mathbb P}^1}(a_i)$. 
Take $E_1= {\mathcal O}_{{\mathbb P}^1}(a_i)$. 
Then the stability condition gives 
$$a_i+ \sum_j f_j({\mathcal O}_{{\mathbb P}^1}(a_i)) \,<\,g_X\, .$$ 
The inequality $f_j(E_1)\ge 0$ implies that $a_i <g_X$, so 
 one has $a_i\le g_X-1$ for all $i$. 
Taking $E_1= \bigoplus_{k\neq i} 
{\mathcal O}_{{\mathbb P}^1}(a_k)$, 
the stability condition gives 
$$
\frac{-a_i}{n-1} + \sum_{j} \frac{f_j(E_1)}{n-1} <g_X\, .$$
By inequality \eqref{(0)}, we have $f_j(E_1) \ge n-2$. Hence 
$$
\frac{-a_i}{n-1} + \frac{g_X(n-2)}{n-1} <g_X\, ,$$ 
hence $-a_i<g_X$ or $a_i \ge -g_X+1$. Thus for all $i$, 
one has $g_X-1 \ge a_i\ge -g_X+1$.
\end{proof} 

\begin{proposition}\label{p4.2}
For a general stable bundle $F$ of rank $n$ 
and degree zero on $X$, the
pullback $E= p^* F$ is the trivial vector bundle of rank $n$.
\end{proposition}

\begin{proof}
We may assume that the pullback $E$ of 
the vector bundle $F$ is of the form 
$$
E\,=\, \bigoplus_{i=1}^s E_i,\,\  
E_i\,=\, {\mathcal O}_{{\mathbb P}^1}(a_i)^{\oplus r_i}\, ,\, 
~ a_1< \cdots < a_s\, .
$$
Two generalized parabolic sheaves
$(E, \{F^1_j(E)\})$ and $(E,\{F^2_j(E)\})$ 
are isomorphic if and only if they differ by an automorphism
of $E$. Isomorphic generalized parabolic sheaves
give isomorphic vector bundles on $Y$. The automorphism group 
${\rm Aut}(E)$ of $E$ is a nonempty Zariski open subset of the vector 
space $Hom(E,E)$. Hence the dimensions of ${\rm Aut}(E)$
and $Hom(E,E)$ coincide. Now 
$$
Hom(E,E)= \bigoplus_{i=1}^s Hom(E_i,E_i) \oplus
\bigoplus_{i<j} Hom(E_i,E_j) \oplus \bigoplus_{i>j} Hom(E_j,E_j)\, .$$ 
Since $Hom({\mathcal O}_{{\mathbb P}^1}(a_i), 
{\mathcal O}_{{\mathbb P}^1}(a_j))=0$ for $i>j$, one has 
dim $Hom(E_i,E_j)=0$ for $i>j$. Also $Hom(E_i,E_i)=M_{r_i}(k)$.

Let $i<j$. Then we have 
$$Hom(E_i,E_j) \,=\,  
Hom({\mathcal O}_{{\mathbb P}^1}(a_i),{\mathcal O}_{{\mathbb P}^1}(a_j))^{\oplus
r_ir_j} \,=\, Hom({\mathcal O}_{{\mathbb P}^1}(a_j-a_i))^{\oplus r_ir_j}\, .
$$ 
Hence $\dim Hom(E_i,E_j)\,=\, r_i r_j (a_j-a_i+1)$. 
Note that $(a_j-a_i+1)\ge 2$. Since $d(E)\,=\,0$, we have $\sum_i a_i 
r_i\,=\,0$. If $E$ is not the trivial vector bundle, then there exist 
$a_i, a_j$, $i<j$, 
such that $a_i\le -1, a_j\ge 1$; therefore, for these $i,j$, one has 
$a_j-a_i+1 \ge 3$,
and hence $\dim Hom(E_i,E_j)\,>\,2r_ir_j$. 
Consequently, $\dim \bigoplus_{i<j} Hom(E_i,E_j)\,>\, 2\sum_{i<j} r_i r_j$. 

Thus we have
$$\dim Hom(E,E)> \sum_i r_i^2 + 2\sum_{i<j} r_i r_j 
= \sum_{i,j} r_i r_j = (\sum_ i r_i)^2 = n^2\, .$$ 

On the other hand, for the trivial vector bundle 
of rank $n$, we have $\dim Hom(E,E)\,=\, n^2$. 

    The moduli space $M_{{\mathbb P}^1,D_{g_X}}(n,0)$
of generalized parabolic sheaves has dimension $n^2(g_X-1)+1$. 
The locus in $M_{{\mathbb P}^1,D_{g_X}}(n,0)$
of generalized parabolic sheaves
with underlying vector bundle $E$ has dimension 
$$\delta_E = n^2 g_X - \ {\rm dim \ (Aut} \ E)\, .$$
The above computations show that $\delta_E = n^2(g_X-1)+1$ if $E$ is the  trivial vector bundle of rank $n$ and $\delta_E < n^2(g_X-1)+1$ if $E$ is not the trivial vector bundle. By Lemma \ref{l4.1}, 
there are only finitely many possible choices of the
underlying vector bundle. It follows that there is a nonempty open subset of 
$M_{{\mathbb P}^1,D_{g_X}}(n,0)$ consisting of generalized parabolic sheaves
with underlying vector bundle $E$ trivial i.e., there a nonempty open subset of the moduli space of stable bundles of rank $n$ 
and degree zero on $X$ consisting of stable vector bundles $F$ with 
pullback $E= p^* F$ the trivial vector bundle of rank $n$. 
\end{proof}

\begin{theorem}\label{t4.3}
Let $X$ be a rational curve of genus $g_X\ge 2$. Then 
there is a nonempty Zariski open subset $S \subset U_{X}^{'s}(n,0)$
such that for any $F\in S$, the pullback $E= p^* F$ 
is the trivial vector bundle of rank $n$. 
Moreover, $$S\,\cong\,({\rm GL}_n(k)\times \cdots \times {\rm 
GL}_n(k))/\!\!/ {\rm PGL}_n(k)\, .$$
\end{theorem}

\begin{proof}
By Proposition \ref{p4.2}, 
for a general stable bundle $F$ of rank $n$ 
and degree zero on $X$, the pullback $E= p^* F$ 
is a trivial vector bundle of rank $n$. 
Recall that for the GPS 
$(E, \{F_j(E)\})\in M_{{\mathbb P}^1,D_{g_X}}(n,0)$, 
for every $j=1,\cdots ,g_X$, the vector space $F_j(E)$ is the graph of 
an isomorphism $f_j\, : E_{x_j} \,\longrightarrow\, E_{z_j}$.
Therefore, $$(f_1,\cdots ,f_{g_X})\, \in\, 
{\rm GL}_n(k)\times \cdots \times{\rm GL}_n(k)$$ (as
$E$ is the trivial vector bundle). In terms of this identification,
the action of the automorphism group ${\rm Aut}(E)$ on the space
of all GPS with the underlying vector bundle $E$ is
the diagonal conjugation action of ${\rm PGL}_n(k)$ on
${\rm GL}_n(k)\times \cdots \times {\rm GL}_n(k)$, i.e., the action 
$$
g\circ (f_1, \cdots, f_{g_X})= (g^{-1}f_1 g, \cdots, g^{-1} f_{g_X} g),
~\,\,~ g\in \ {\rm PGL}_n(k)\, .
$$ 
Let $S''$ be the open subset of 
$M_{{\mathbb P}^1,D_{g_X}}(n,0)$ corresponding to all GPS whose 
underlying
vector bundle is trivial and the sheaf on $X$ corresponding to
it is a vector bundle. 
Note that $F_j(E)$ is the graph of an isomorphism 
if and only if the projections of $F_j(E)$ to
$E_{x_j}$ and $E_{z_j}$ are both isomorphisms. Therefore,
it follows that $S''$ is isomorphic to
$$
({\rm GL}_n(k)\times \cdots \times {\rm GL}_n(k))/\!\!/
{\rm PGL}_n(k)\, .
$$ 
Since the map $\Phi$ in \eqref{Phi} is an isomorphism over stable 
bundles, it follows that there is an open subset $S \subset 
U_{X}^{'s}(n,0)$ such that $S''$ maps isomorphically onto $S$.
\end{proof}

In view of Theorem \ref{t4.3}, 
to prove that $U^{'s}_X(n,0)$ is rational (respectively, stably rational) 
it suffices to check that the quotient
$({\rm GL}_n(k)\times \cdots \times {\rm GL}_n(k))/\!\!/ {\rm PGL}_n(k)$ 
is (respectively, stably rational). From
Remark \ref{reduction} we conclude that if the following well-known 
question has an affirmative answer, then Theorem \ref{rational21} 
(respectively, Theorem \ref{stablyrational}) in fact holds for arbitrary 
$h$.

\begin{question}\label{c4.3}
For ${\rm PGL}_n(k)$ acting diagonally on ${\rm GL}_n(k)\times \cdots \times 
{\rm GL}_n(k)$ by the conjugation action, is the quotient 
$({\rm GL}_n(k)\times \cdots \times {\rm GL}_n(k))/\!\!/ {\rm PGL}_n(k)$ 
rational (respectively, stably rational)?
\end{question}

Note that ${\rm GL}_n(k)\times \cdots \times {\rm GL}_n(k)$ 
is an open subset of ${\rm M}_n(k)\times \cdots \times {\rm M}_n(k)$ 
and the conjugation action of ${\rm PGL}_n(k)$ extends to 
a conjugation action on ${\rm M}_n(k)\times \cdots \times {\rm M}_n(k)$. 
Thus $({\rm GL}_n(k)\times \cdots \times {\rm GL}_n(k))/\!\!/ {\rm PGL}_n(k)$ 
is an open subset of 
$({\rm M}_n(k)\times \cdots \times {\rm M}_n(k))/\!\!/ {\rm PGL}_n(k)$.
The question of rationality and stable rationality of the latter 
has been studied extensively by many algebraists and algebraic geometers. 
The answer is known in some cases but the general case is still open. 

\subsection{Generic matrices}
Let $K$ be a field of characteristic $0$, $n$ a positive integer 
and $m \in {\mathbb N}$. Let 
$ \{x_{ij}^m \mid 1 \le i,j \le n \}$ be independent commuting 
indeterminates over $K$. The matrix $X_m$ whose $i,j$th entry is 
$x_{i,j}^m$ is called an $n\times n$ generic matrix over $K$. 
Thus $X_m = (x_{i,j}^m) \in M_n(K[x_{i,j}^m])$ (\cite{F3}, \cite{S}). 
The $K$-subalgebra of $M_n(K[x_{i,j}^m])$ generated by the 
$X_m$ is called the ring of generic matrices. 
Subrings of this ring are closely related to the 
invariants for the action of ${\rm GL}_n$ on the products of $M_n$'s.

\begin{proposition}\label{c4.5}
Let $Y$ be a rational nodal curve of genus $2$.
If $h\,\le\, 4$, then $U_Y(n,d)$ is rational.
\end{proposition}

\begin{proof}
For $n=2$, it is well known that the function field 
$${\mathbb C}(M_2\times M_2)^{{\rm PGL}_2}\,=\, 
{\mathbb C}(Tr(X), Tr(Y), Tr(X^2), Tr(Y^2), Tr(XY))\, ,$$ 
where $X$ and $Y$ are generic matrices corresponding to 
the two components and hence is a rational function field
(cf. \cite{Re}). Formanek proved the rationality 
of ${\mathbb C}(M_n\times M_n)^{{\rm PGL}_n}$ for $n=3$ \cite{F1}, 
and $n=4$ \cite{F2}. By Theorem \ref{t4.3}, 
this proves that $U^{'s}_X(h,0)$ is rational for $h\le 4$. 
The proposition now follows from Theorem \ref{birational}.
\end{proof}

This finishes the proof of Theorem \ref{rational21}.

\subsection{Stable Rationality and Retract Rationality}

We recall that a variety $Z$ is called \textit{stably rational} if 
$X \times {\mathbb P}^m$ is rational for some positive integer $m$.

\begin{definition}\label{retract}
{\rm Let $f\,: \,X \,\dasharrow \, Y$ be a dominant rational map of irreducible algebraic varieties. 
Then $Y$ is called} retract of $X$ via $f$ {\rm if for every (not necessarily 
irreducible) subvariety $X_0 \neq X$, there is a rational section 
$s\, :\, Y\, \dasharrow \, X$ of $f$ such that the image $f(Y)$ is not contained in $X_0$.}

{\rm The variety $Y$ is called a} retract {\rm of $X$ if $Y$ is retract of $X$ via some
rational dominant map.}
\end{definition}

\begin{remark}
{\rm If we view $X$ as a variety over $k(Y)$, then Definition 
\ref{retract} can be restated as follows: the variety $Y$ is retract of $X$ if and only if 
$k(Y)$--points are dense in $X$.}

{\rm The property of being retract depends only on the field extension of 
the function fields of $X$ and $Y$. In other words, it does not change if $X$ and
$Y$ are replaced by birationally isomorphic varieties.}   
\end{remark}

\begin{definition}\label{retractrational}
{\rm A variety $Z$ is called} retract rational {\rm if $Z$ is a retract of ${\mathbb A}^m_k$,
and also ${\mathbb A}^m_k$ is a retract of $Z$ for some positive integer $m$.} 
\end{definition}

\begin{theorem} (Theorem \ref{stablyrational})
Let $Y$ be a rational nodal curve with $g_Y\,\ge\, 1$. 
\begin{enumerate}
	\item For $h$ dividing $420= 3.4.5.7$, 
the moduli space $U_Y(n,d)$ is stably rational.
  \item For $h$ a prime number, $U^{'s}_Y(n,d)$ is retract rational. 
\end{enumerate}
\end{theorem}

\begin{proof}
Bessenrodt and Le Bruyn proved that 
${\mathbb C}(M_n\times M_n)^{PGL_n}$ is stably 
rational for $n=5,7$ (\cite[Theorem 1]{BB}, see also \cite{B}). They
used this to deduce that if $V$ is a finite dimensional complex 
representation of ${\rm PGL}_n$ such that ${\rm PGL}_n$ acts almost freely 
(i.e., there is a Zariski open subset of V such that every point 
of this open set has a trivial stabilizer), then for any $h$ 
dividing $420= 3.4.5.7$, the field of invariants 
${\mathbb C}(V)^{{\rm PGL}_n}$ is stably rational \cite[Theorem 2]{BB}. 

It is easy to see that the action of ${\rm PGL}_h$ 
on $M_h\times M_h$ is almost free and hence the action 
on $M_h\times \cdots \times M_h$ (a $g$-fold product with $g\ge 2$) 
is also almost free. Hence by \cite[Theorem 2]{BB}, for any $h$ 
dividing $420= 3.4.5.7$, the field of invariants 
${\mathbb C}(M_h \times \cdots \times M_h)^{{\rm PGL}_h}$ is stably
rational. The first part of the theorem now follows by the
argument in the proof of Proposition \ref{c4.5}.

      Using \cite[Corollary 5.3, p. 211]{S1}, one sees that   
if $h$ is a prime number then the moduli space $U^{'s}_Y(h,0)$ is retract rational. It now follows from Theorem \ref{birational}(1) that for $h$ a prime number, $U^{'s}_Y(n,d)$ is retract rational. 

\end{proof}

\end{document}